\documentclass[11pt,leqno,english]{article}
\usepackage{amssymb}
\usepackage[reqno]{amsmath}
\usepackage{amsfonts}
\usepackage{geometry}

\newtheorem{theorem}{Theorem}

\newtheorem{corollary}{Corollary}

\newtheorem{definition}[theorem]{Definition}

\newtheorem{lemma}{Lemma}

\newenvironment{proof}[1][Proof]{\noindent\textbf{#1.} }{\ \rule{0.5em}{0.5em}}
\input{tcilatex}
\begin{document}

\title{ An hybrid deterministic-stochastic iterative procedure to solve the
heat equation}
\author{Fouad Maouche\thanks{%
Laboratory of Applied Mathematics,Faculty of Exact Sciences, University of
Bejaia, Algeria. E-mail: fouad.maouche@univ-bejaia.dz} \and Laboratory of
Applied Mathematics, \and Faculty of Exact Sciences, University of Bejaia,
Algeria.}
\date{}
\maketitle

\begin{abstract}
Our goal in this paper is to solve the 1-D heat equation by an hybrid
deterministic-stochastic iterative procedure . The deterministic side
consists in discretizing the equation by the Crank-Nicolson method and the
stochastic side consists of applying Robbins Monro procedure to solve the
resulting matrix system. The almost complete convergence and the rate of
convergence of our procedure are established.

\textbf{Keywords: }Robbins-Monro, heat equation, Crank-Nicolson, Almost
complete convergence.

\textbf{Mathematics Subject Classification}: 65N21, 65C50.
\end{abstract}

\section*{Introduction}

Consider the heat equation problem on an interval $I$ in $\mathbb{R}$,

\begin{equation}
\left\{ 
\begin{array}{c}
\frac{^{\partial u(x,t)}}{\partial t}-D\frac{^{\partial ^{2}u(x,t)}}{%
\partial x^{2}}=0, \\ 
u(x,0)=f(x).%
\end{array}%
\right.  \label{a}
\end{equation}

Many researchers have worked on the one dimensional heat conduction equation
using various numerical methods and finite difference methods are the mostly
used of all the numerical methods \textrm{\cite{sun,sun1}}. Recently, O.
Nikan et al \textrm{\cite{nik}} proposed an efficient coupling of the
Crank--Nicolson scheme and localized meshless technique for viscoelastic
wave model in fluid flow, C.Chen et al \textrm{\cite{che} } present a
second-order accurate Crank-Nicolson scheme for the two-grid finite elements
for nonlinear Sobolev equations and M.\ Ran et al \textrm{\cite{run}}\
proposed a linearized Crank--Nicolson scheme for the nonlinear time space
fractional Schr\"{o}dinger equations.

Many other authors proposed iterative methods to solve the heat conduction
equation. For example, Newell--Whitehead--Segel equations of fractional
order are solved by fractional variational iteration method in \textrm{\cite%
{pra} }and Landweber iterative regularization method to identify the initial
value problem of the time-space fractional diffusion-wave equation was
studied in \textrm{\cite{yan}.}

Stochastic algorithms are part of modern techniques for numerical solution
of many practical problems : signal processing and adaptive control \cite%
{Ding,Kushner}, inverse problems \cite{maouche}, communication and system
identification \textrm{\cite{wang,liu,ding3}}.

Let $(\Omega ,\digamma ,P)$ be a probability space and assume that $\inf
\left\{ re\text{ }\lambda ;\text{ }\lambda \in \sigma (A)\right\} >0$ with $%
\sigma (A)$ is a spectrum of a matrix $A.$

Our goal is to solve the heat equation $\left( \ref{a}\right) $ by a hybrid
iterative procedure (deterministic and stochastic). The deterministic side
consists in discretizing equation $\left( \ref{a}\right) $ by the
Crank-Nicolson method and the stochastic side consists of applying Robbins
Monro procedure \textrm{\cite{Robbins}}, which uses the full forward model
when the noise on the right hand side is stochastic, to solve the resulting
matrix system.

By Hoeffding exponential inequalities, the almost complete convergence and
the rate of convergence of the iterative procedure to solve the heat
equation $\left( \ref{a}\right) $ are established.

\section{Methodology}

The crank-Nicolson method combines the stability of an implicit method with
the accuracy of a second-order method in both space and time. Simply from
the average of the explicit and implicit FTCS schemes (left and right sides
are centred at time step $m+0.5$).

The following approximation expression holds

\begin{equation*}
\frac{^{\partial u(x,t)}}{\partial t}-D\frac{^{\partial ^{2}u(x,t)}}{%
\partial x^{2}}\simeq \frac{u_{n,m+1}-u_{n,m}}{\Delta t}-\frac{D}{2}\left( 
\frac{u_{n+1,m+1}-2u_{n,m+1}+u_{n-1,m+1}+u_{n+1,m}-2u_{n,m}+u_{n-1,m}}{%
\Delta x^{2}}\right) .
\end{equation*}

The scheme can be written as follow

\begin{equation*}
au_{n-1,m}+(2-2a)u_{n,m}+au_{n+1,m}=-au_{n-1,m+1}+(2+2a)u_{n,m+1}-au_{n+1,m+1}.
\end{equation*}

These equations holds for $1\leq n\leq N-1,$ and the boundary conditions
supply the two missing equations. The Crank-Nicolson method will be written
in the following matrix form

$%
\begin{pmatrix}
-a & 2+2a & -a & 0 & . & . & . \\ 
0 & -a & 2+2a & -a & . & . & . \\ 
. & . & . & . & . & . & . \\ 
. & . & . & . & . & . & . \\ 
. & . & . & . & . & . & . \\ 
. & . & . & . & 2+2a & -a & 0 \\ 
. & . & . & . & -a & 2+2a & -a%
\end{pmatrix}%
\begin{pmatrix}
u_{0,m+1} \\ 
u_{1,m+1} \\ 
. \\ 
. \\ 
. \\ 
u_{N-1,m+1} \\ 
u_{N,m+1}%
\end{pmatrix}%
=$

$%
\begin{pmatrix}
a & 2-2a & a & 0 & . & . & . \\ 
0 & a & 2-2a & a & . & . & . \\ 
. & . & . & . & . & . & . \\ 
. & . & . & . & . & . & . \\ 
. & . & . & . & . & . & . \\ 
. & . & . & . & 2-2a & a & 0 \\ 
. & . & . & . & a & 2-2a & a%
\end{pmatrix}%
\begin{pmatrix}
u_{0,m} \\ 
u_{1,m} \\ 
. \\ 
. \\ 
. \\ 
u_{N-1,m} \\ 
u_{N,m}%
\end{pmatrix}%
.$

The matricial equation above is a representation of $N-1$ equations with $%
N+1 $ unknown, the matrices have $N-1$ rows and $N+1$columns and the two
missing equations come from the boundary conditions which we use to convert
this matricial equation into a system of equations involving a square matrix
of the form:

\begin{equation}
A_{m+1}u_{m+1}+r_{m+1}=B_{m}u_{m}+w_{m},  \label{ck1}
\end{equation}

with

$A_{m+1}=%
\begin{pmatrix}
2+2a & -a & 0 & . & . & . & . \\ 
-a & 2+2a & . & . & . & . & . \\ 
0 & . & . & . & . & . & . \\ 
. & . & . & . & . & . & . \\ 
. & . & . & . & . & -a & . \\ 
. & . & . & . & -a & 2+2a & -a \\ 
. & . & . & . & 0 & -a & 2+2a%
\end{pmatrix}%
,$

$B_{m}=%
\begin{pmatrix}
2-2a & a & 0 & . & . & . & . \\ 
a & 2-2a & a & . & . & . & . \\ 
0 & . & . & . & . & . & . \\ 
. & . & . & . & . & . & . \\ 
. & . & . & . & . & a & . \\ 
. & . & . & . & a & 2-2a & a \\ 
. & . & . & . & 0 & a & 2-2a%
\end{pmatrix}%
,$

$u_{m+1}=%
\begin{pmatrix}
u_{1,m+1} \\ 
. \\ 
. \\ 
. \\ 
u_{N-1,m+1}%
\end{pmatrix}%
,$ $u_{m}=%
\begin{pmatrix}
u_{1,m} \\ 
. \\ 
. \\ 
. \\ 
u_{N-1,m}%
\end{pmatrix}%
,$ $r_{m+1}=%
\begin{pmatrix}
-au_{0,m+1} \\ 
0 \\ 
. \\ 
. \\ 
-au_{N,m+1}%
\end{pmatrix}%
$ and $w_{m}=%
\begin{pmatrix}
au_{0,m} \\ 
0 \\ 
. \\ 
. \\ 
a_{N,m}%
\end{pmatrix}%
.$

The equation $\left( \ref{ck1}\right) $ can be written in following form

\begin{equation}
A_{m+1}u_{m+1}=B_{m}u_{m}+w_{m}-r_{m+1}.  \label{ck2}
\end{equation}

The matrix inversion method to find the solution $u_{m+1}$ is very time
consuming and computationally inefficient. In this paper, we propose the
following iterative procedure to solve the equation $\left( \ref{ck2}\right)
.$

\begin{equation}
\left\{ 
\begin{array}{c}
X_{(k+1)}=X_{(k)}-\frac{1}{k}\left[
A_{m+1}X_{(k)}-(B_{m}u_{m}+w_{m}-r_{m+1})-\xi _{(k)}\right] ,\text{ }%
k=1,2,..., \\ 
X_{(0)}\in 
\mathbb{R}
^{N-1}%
\end{array}%
\right.  \label{cA}
\end{equation}

where $(X_{(k)})_{k\in 
\mathbb{N}
^{\ast }}$ are vectors of $%
\mathbb{R}
^{n}$and $\left( \xi _{(k)}\right) _{k\in 
\mathbb{N}
^{\ast }}$ is a sequence of bounded and i.i.d random variables with values
in $%
\mathbb{R}
^{n}$ satisfying 
\begin{equation}
\left\Vert \xi _{(k)}\right\Vert <b,\text{ }b\in 
\mathbb{R}
\label{a'A}
\end{equation}

\section{ Preliminary results}

According to \textrm{\cite[lemma 1]{maouche} }and after successive
iterations, the following relation is obtained%
\begin{equation}
X_{(k+1)}-X_{(ex)}=\dprod\limits_{i=1}^{k}(I-\frac{1}{i}%
A_{m+1})(X_{1}-X_{ex})+\sum_{i=1}^{k}\dprod\limits_{j=i+1}^{k}(I-\frac{1}{j}%
A_{m+1})\frac{1}{i}\xi _{i}.  \label{e}
\end{equation}

Where, $X_{(ex)}$ is the exact solution and $\dprod\limits_{j=k+1}^{k}(I-%
\frac{1}{j}A_{m+1})=I$ for $1\leq i,j\leq k,$ with $I$ is the unit matrix.

\begin{lemma}
Suppose that $\inf \left\{ re\text{ }\lambda ;\text{ }\lambda \in \sigma
(A_{m+1})\right\} >0.$ The following expression holds. 
\begin{equation}
\exists \gamma >0,\text{ }\exists p>0,\text{ }\forall \text{ }1\leq i\leq
k:\left\Vert \dprod\limits_{j=i+1}^{k}(I-\frac{1}{j}A_{m+1})\right\Vert \leq
\gamma \left( \frac{i+1}{k+1}\right) ^{p}.  \label{ee}
\end{equation}

\ 
\end{lemma}

\begin{proof}
Under the condition $\inf \left\{ re\text{ }\lambda ;\text{ }\lambda \in
\sigma (A_{m+1})\right\} >0,$ H. Walk obtained the following result \cite[%
Lemma 3.b]{walk1}\newline
\begin{equation}
\exists \gamma >0,\text{ }\exists p>0,\text{ }\forall \text{ }1\leq i\leq
k:\left\Vert \dprod\limits_{j=i+1}^{k}(I-\frac{1}{j}A_{m+1})\right\Vert \leq
\gamma \left( \dprod\limits_{j=i+1}^{k}\left( 1-\frac{1}{j}\right) \right)
^{p}.  \label{walk}
\end{equation}%
\newline
Then, taking into account that: $\ln \left( 1+x\right) \leq x$, for $x>-1$,
we have%
\begin{equation*}
p\ln \left( 1-\frac{1}{j^{\theta }}\right) \leq -\frac{p}{j^{\theta }}
\end{equation*}%
\begin{equation*}
\sum_{j=i+1}^{k}\ln \left( 1-\frac{1}{j^{\theta }}\right) ^{p}\leq
\sum_{j=i+1}^{k}-\frac{p}{j^{\theta }}=-p\sum_{j=i+1}^{k}\frac{1}{j^{\theta }%
}=-p\int\limits_{i+1}^{k+1}\frac{1}{x^{\theta }}dx.
\end{equation*}%
\newline
Then%
\begin{equation*}
\sum_{j=i+1}^{k}\ln \left( 1-\frac{1}{j^{\theta }}\right) ^{p}\leq p\ln (%
\frac{i+1}{k+1}).
\end{equation*}%
\newline
This implies that%
\begin{equation*}
\exp \left( \sum_{j=i+1}^{k}\ln \left( 1-\frac{1}{j^{\theta }}\right)
^{p}\right) \leq \exp \left( \ln \left( \frac{i+1}{k+1}\right) ^{p}\right) .
\end{equation*}%
\newline
So,%
\begin{equation}
\gamma \left( \dprod\limits_{j=i+1}^{k}\left( 1-\frac{1}{j^{\theta }}\right)
\right) ^{p}\leq \gamma \left( \frac{i+1}{k+1}\right) ^{p}.  \label{gA}
\end{equation}%
\newline
\newline
From $(\ref{gA})$ we deduce that%
\begin{equation}
\lim_{k\rightarrow \infty }\left\Vert \dprod\limits_{j=1}^{k}(I-\frac{1}{j}%
A_{m+1})\right\Vert \leq \lim_{k\rightarrow \infty }\gamma \left(
\dprod\limits_{j=1}^{k}\left( 1-\frac{1}{j^{\theta }}\right) \right) ^{p}=0.
\label{ine}
\end{equation}
\end{proof}

\begin{lemma}
\bigskip Under assumptions of Lemma 1, the following expression holds.%
\begin{equation}
\exists \gamma >0,\text{ }\exists p>0,\text{ }\forall \text{ }1\leq i\leq
k:\sum_{i=1}^{k}\left\Vert \dprod\limits_{j=i+1}^{k}(I-\frac{1}{j}A_{m+1})%
\frac{1}{i}\right\Vert ^{2}\leq C\frac{\gamma ^{2}}{(k+1)^{2p}},  \label{h'}
\end{equation}%
with $C$ is a constant$.$\newline
\end{lemma}

\begin{proof}
By virtue of the relation $(\ref{ee})$ one has%
\begin{equation*}
\exists \gamma >0,\text{ }\exists p>0,\text{ }\forall \text{ }1\leq i\leq
k:\left\Vert \dprod\limits_{j=i+1}^{k}(I-\frac{1}{j}A_{m+1})\frac{1}{i}%
\right\Vert ^{2}\leq \gamma ^{2}\frac{(i+1)^{2p}}{(k+1)^{2p}i^{2}}.
\end{equation*}%
\newline
Then\newline
\begin{equation}
\exists \gamma >0,\text{ }\exists p>0,\text{ }\forall \text{ }1\leq i\leq
k:\sum_{i=1}^{k}\left\Vert \dprod\limits_{j=i+1}^{k}(I-\frac{1}{j}A_{m+1})%
\frac{1}{i}\right\Vert ^{2}\leq \frac{\gamma ^{2}}{(k+1)^{2p}}\sum_{i=1}^{k}%
\frac{(i+1)^{2p}}{i^{2}}.  \label{1}
\end{equation}%
\newline
By Kronecker's lemma, $\frac{\gamma ^{2}}{(k+1)^{2p}}\sum_{i=1}^{k}\frac{%
(i+1)^{2p}}{i^{2}}$ tends to $0$ when $k$ tends to infinity.\newline
In fact, $\left( \frac{1}{i^{2}}\right) _{i\in 
\mathbb{N}
^{\ast }}$ is a convergent sequence and $\lim\limits_{i\rightarrow +\infty
}(i+1)^{2p}=+\infty .$\newline
So, 
\begin{equation}
\lim\limits_{i\rightarrow +\infty }\frac{1}{(k+1)^{2p}}\sum_{i=1}^{k}\frac{%
(i+1)^{2p}}{i^{2}}=0.  \label{kro}
\end{equation}%
\newline
From the relation $(\ref{kro})$ one deduces that: $\exists (N^{\ast }$ $\in 
\mathbb{N}
^{\ast })$\ such that 
\begin{equation}
\lim\limits_{k\rightarrow \infty }\frac{1}{\left( k+1\right) ^{2p}}%
\sum_{i=N^{\ast }+1}^{k}\frac{(i+1)^{2p}}{i^{2}}=0.
\end{equation}%
\newline
We have the following relationship%
\begin{equation*}
\frac{1}{\left( k+1\right) ^{2p}}\sum_{i=1}^{k}\frac{(i+1)^{2p}}{i^{2}}=%
\frac{1}{\left( k+1\right) ^{2p}}\sum_{i=1}^{N^{\ast }}\frac{(i+1)^{2p}}{%
i^{2}}+\frac{1}{\left( k+1\right) ^{^{2c}}}\sum_{i=N^{\ast }+1}^{k}\frac{%
(i+1)^{2p}}{i^{2}}
\end{equation*}%
\newline
\begin{equation*}
\leq \frac{1}{\left( k+1\right) ^{2p}}\sum_{i=1}^{N^{\ast }}\frac{(i+1)^{2p}%
}{i^{2}}=\frac{C}{\left( k+1\right) ^{2p}},
\end{equation*}%
\newline
with, $\sum_{i=1}^{N}\frac{(i+1)^{2ap}}{i^{2}}=C.$\newline
Replacing in $(\ref{1})$ we find $(\ref{h'}).$\newline
\end{proof}

\section{Exponential inequalities and convergence results}

In this section, exponential inequalities of the Hoeffding type are
established. These allow us to establish the almost complete convergence and
the rate of convergence of the iterative procedure $\left( \ref{cA}\right) $
to solve the heat equation $\left( \ref{a}\right) .$

\begin{definition}
\bigskip The sequence of random variables $\left( X_{(k)\text{ }}\right)
_{k\in 
\mathbb{N}
^{\ast }}$ converges almost completely (a.co) to a random variable $X,$ when 
$k$ tends to infinity, if and only if: $\forall \varepsilon
>0,\sum_{k=1}^{+\infty }\mathbb{P}\left( \left\Vert
X_{(k+1)}-X_{(ex)}\right\Vert >\varepsilon \right) <+\infty .$
\end{definition}

\begin{theorem}
Let $A\in L(H).$Under the condition $\inf \left\{ re\text{ }\lambda ;\text{ }%
\lambda \in \sigma (A)\right\} >0$, the following exponential inequalities
hold.\newline
\begin{equation}
\mathbb{P}\left( \left\Vert X_{(k+1)}-X_{(ex)}\right\Vert >\varepsilon
\right) \leq 2\exp \left( -\frac{(k+1)^{2p}\varepsilon ^{2}}{\alpha }\right)
,\text{ }\alpha \in 
\mathbb{R}
.  \label{n}
\end{equation}
\end{theorem}

\begin{proof}
By virtue of the relation $(\ref{e}),$ one has the following expression%
\begin{equation*}
\mathbb{P}\left( \left\Vert X_{(k+1)}-X_{(ex)}\right\Vert >\varepsilon
\right) =P\left( \left\Vert \dprod\limits_{i=1}^{k}(I-\frac{1}{i}%
A_{m+1})(X_{(1)}-X_{(ex)})+\sum_{i=1}^{k}\dprod\limits_{j=i+1}^{k}(I-\frac{1%
}{j}A_{m+1})\frac{1}{i}\xi _{i}\right\Vert >\varepsilon \right)
\end{equation*}%
\newline
\begin{equation*}
\leq \mathbb{P}\left( \left\Vert \sum_{i=1}^{k}\dprod\limits_{j=i+1}^{k}(I-%
\frac{1}{j}A_{m+1})\frac{1}{i}\xi _{i}\right\Vert >\varepsilon -\left\Vert
\dprod\limits_{i=1}^{k}(I-\frac{1}{i}A_{m+1})(X_{(1)}-X_{(ex)})\right\Vert
\right) .
\end{equation*}%
\newline
The relation $(\ref{ine})$ proves that: 
\begin{equation}
\exists \varepsilon >0,\left\Vert \dprod\limits_{i=1}^{k}(I-\frac{1}{i}%
A_{m+1})(X_{(1)}-X_{(ex)})\right\Vert \leq \frac{\varepsilon }{2}.
\end{equation}%
\newline
Then,\newline
\begin{equation*}
\mathbb{P}\left( \left\Vert X_{(k+1)}-X_{(ex)}\right\Vert >\varepsilon
\right) \leq \mathbb{P}\left( \left\Vert
\sum_{i=1}^{k}\dprod\limits_{j=i+1}^{k}(I-\frac{1}{j}A_{m+1})\frac{1}{i}\xi
_{i}\right\Vert >\varepsilon -\frac{\varepsilon }{2}\right)
\end{equation*}%
\newline
\begin{equation*}
\leq \mathbb{P}\left( \left\Vert \sum_{i=1}^{k}\dprod\limits_{j=i+1}^{k}(I-%
\frac{1}{j}A_{m+1})\frac{1}{i}\xi _{i}\right\Vert >\frac{\varepsilon }{2}%
\right) .
\end{equation*}%
\newline
We pose 
\begin{equation}
\eta _{i}=\dprod\limits_{j=i+1}^{k}(I-\frac{1}{j}A_{m+1})\frac{1}{i}\xi _{i}.
\label{theta}
\end{equation}%
\newline
$\left( \eta _{i}\right) _{i\in 
\mathbb{N}
^{\ast }}$ is a sequence of bounded and i.i.d random variables in a Hilbert
space such that%
\begin{equation}
\left\Vert \eta _{i}\right\Vert <\left\Vert \dprod\limits_{j=i+1}^{k}(I-%
\frac{1}{j}A_{m+1})\right\Vert \frac{1}{i}b=d_{i}.  \label{ntheta}
\end{equation}%
\newline
Thus%
\begin{equation}
\mathbb{P}\left( \left\Vert X_{(k+1)}-X_{(ex)}\right\Vert >\varepsilon
\right) \leq \mathbb{P}\left( \left\Vert \sum_{i=1}^{k}\eta _{i}\right\Vert >%
\frac{\varepsilon }{2}\right) .  \label{1111}
\end{equation}%
We give the Pinelis-Hoeffding inequality for the sequence $\left( \eta
_{i}\right) _{i\in 
\mathbb{N}
^{\ast }}$ such that $\left\Vert \eta _{i}\right\Vert <d_{i}$ in a Hilbert
space $H$ ( \cite{pinelis1}).%
\begin{equation}
\mathbb{P}\left( \left\Vert \sum_{i=1}^{k}\eta _{i}\right\Vert >\varepsilon
\right) \leq 2\exp \left( -\frac{\varepsilon ^{2}}{2\sum_{i=1}^{k}d_{i}^{2}}%
\right) .  \label{xx}
\end{equation}%
Then, we deduce from $(\ref{ntheta})$ and $(\ref{xx})$ the following
relation.%
\begin{equation}
\mathbb{P}\left( \left\Vert \sum_{i=1}^{k}\eta _{i}\right\Vert >\frac{%
\varepsilon }{2}\right) \leq 2\exp \left( -\frac{\varepsilon ^{2}}{%
8b^{2}\sum_{i=1}^{k}\left\Vert \dprod\limits_{j=i+1}^{k}(I-\frac{1}{j}%
A_{m+1})\right\Vert ^{2}(\frac{1}{i})^{2}}\right) .  \label{PIN}
\end{equation}%
Finally, by virtue of the relation $(\ref{1111})$ and $(\ref{h'})$ and from
the relation $(\ref{PIN})$, we obtain%
\begin{equation*}
\mathbb{P}\left( \left\Vert X_{(k+1)}-X_{(ex)}\right\Vert >\varepsilon
\right) \leq 2\exp \left( -\frac{\varepsilon ^{2}}{\frac{8C\left( \gamma
b\right) ^{2}}{(k+1)^{2p}}}\right) .
\end{equation*}%
By putting $\alpha =8C\left( \gamma b\right) ^{2}$ we will find $(\ref{n}).$%
\newline
\end{proof}

In the next corollary, we give the proof that the iterative procedure $(\ref%
{cA})$ converges almost completely to the solution of the equation $(\ref{a}%
).$

\begin{corollary}
Under the conditions of Theorem 1:

The recursive procedure $(\ref{cA})$ converges almost completely (a.co) to
the solution of the equation$\ (\ref{a}):$%
\begin{equation}
\forall \varepsilon >0,\text{ }\sum_{k=1}^{+\infty }\mathbb{P}\left(
\left\Vert X_{(k+1)}-X_{(ex)}\right\Vert >\varepsilon \right) <+\infty .
\label{con}
\end{equation}

Additionally,%
\begin{equation}
\left\Vert X_{(k+1)}-X_{(ex)}\right\Vert =O(k^{-2p}),\text{ }p>\frac{1}{2}.
\label{vcon}
\end{equation}
\end{corollary}

\begin{proof}
1) Let us pose 
\begin{equation*}
v_{k}=2\exp \left( -\frac{(k+1)^{2p}\varepsilon ^{2}}{\alpha }\right) \leq
2\exp \left( -\left( k+1\right) ^{2p}\varepsilon ^{2}\right) ,
\end{equation*}%
\newline
\newline
Applying Cauchy's rule to the positive term series $v_{k}$ it follows that:%
\newline
When $p>\frac{1}{2}$, $\sum_{k=1}^{+\infty }v_{k}$ is a convergent series.%
\newline
This implies that%
\begin{equation}
\forall \varepsilon >0,\sum_{k=1}^{+\infty }\mathbb{P}\left( \left\Vert
X_{(k+1)}-X_{(ex)}\right\Vert >\varepsilon \right) \leq \sum_{k=1}^{+\infty
}2\exp \left( -\left( k+1\right) ^{2p}\varepsilon ^{2}\right) <+\infty .
\label{ici}
\end{equation}%
\newline
2) To obtain $(\ref{vcon}),$ it is sufficient to choose $A=\varepsilon
k^{2p} $ in $(\ref{ici})$ to have%
\begin{equation*}
\sum_{k=1}^{+\infty }P\left( \left\Vert X_{(k+1)}-X_{(ex)}\right\Vert
>Ak^{-2p}\right) <+\infty .
\end{equation*}
\end{proof}

\bigskip

${\LARGE Acknowledgement:}$ We acknowledge support of "Direction de la
Recherche Scientifique et du D\'{e}veloppement Technologique
DGRSDT".MESRS,Algeria.

\end{document}